\documentclass[12pt]{amsart}
\usepackage{amsmath,amsthm,amssymb}
\usepackage{mathrsfs}
\usepackage{cite}

\usepackage{amsmath,amsthm,amssymb}
\usepackage{mathrsfs}
\usepackage{cite}

\usepackage{graphicx}
\usepackage{amsmath,amssymb}

\usepackage{color}

\newtheorem{theorem}{Theorem}

\newtheorem{corollary}{Corollary}

\newtheorem{example}{Example}

\newtheorem{lemma}{Lemma}

\numberwithin{equation}{section}

\allowdisplaybreaks
\begin{document}
\title[Rotationally typically  real logharmonic
mappings ]{Characterizing rotationally typically real logharmonic
mappings}

\author[N. M. Alarifi]{Najla M. Alarifi}
\address{Department of Mathematics, Imam Abdulrahman Bin Faisal University, Dammam 31113, Kingdom of Saudi Arabia}
\email{najarifi@hotmail.com}

\author[Z. Abdulhadi]
{Zayid AbdulHadi}
\address{Department of Mathematics\\
American University of Sharjah\\
Sharjah, Box 26666\\
UAE}
\email{zahadi@aus.edu}

\author[R. M. Ali]{Rosihan M. Ali}
\address{School of Mathematical Sciences, Universiti Sains Malaysia, 11800
USM, Penang, Malaysia}
\email{rosihan@usm.my}

\begin{abstract}
This paper treats the class of normalized logharmonic mappings $f(z)=zh(z)\overline{g(z)}$ in the unit disk satisfying $\varphi(z)=zh(z)g(z)$ is analytically typically real. Every such mapping $f$ is shown to be a product of two particular logharmonic mappings, each of which admits an integral representation. Also obtained is the radius of starlikeness and an upper estimate for arclength.
Additionally, it is shown that $f$ maps the unit disk into a domain symmetric with respect to the real axis when it is univalent and its second dilatation has real coefficients.
\end{abstract}
\maketitle
\section{Introduction}
Let $\mathcal{H}(U)$ be the linear space of analytic functions defined in the unit
disk $U=\{z: |z|<1\}$ of the complex plane $\mathbb{C}.$ Let $B$ denote the set
of self-maps $a\in $ $\mathcal{H}(U),$ and
 $B_0$ its subclass consisting of $a \in B(U)$ with $a(0)=0.$
A logharmonic
mapping in $U$ is a solution of the nonlinear elliptic partial
differential equation
\begin{equation}\label{ede}
\overline{\left( \frac{f_{\overline{z}}(z)}{f(z)} \right)}=a(z)\frac{f_{z}(z)}{f(z)},\vspace{+0.1cm}
\end{equation}%
where the second dilatation function $a$ lies in $B$. Thus the Jacobian\vspace{+0.1cm}
\begin{equation*}
J_{f}=\left\vert f_{z}\right\vert ^{2}(1-|a|^{2})\vspace{+0.1cm}
\end{equation*}
is positive, and all non-constant  logharmonic mappings are
sense-preserving and open in $U$.

\vspace{+0.2cm}
If $f$ is a non-constant logharmonic
mapping which vanishes only at $z=0$, then \cite{Abd1} $f$ admits the
representation
\begin{equation}  \label{eq1.2}
f(z)=z^{m}|z|^{2\beta m}h(z)\overline{g(z)},
 \vspace{+0.2cm}
\end{equation}
where $m$ is a positive integer, ${\rm Re\ } \beta >-1/2$, and $h, g \in \mathcal{H}(U)$ satisfy $g(0)=1$ and $h(0)\neq 0.$ The exponent $\beta $ in \eqref{eq1.2} depends only on $a(0)$ and
is given by
\begin{equation*}
\beta =\overline{a(0)}\dfrac{1+a(0)}{1-|a(0)|^{2}}.
 \vspace{+0.2cm}
\end{equation*}
Note that $f(0)\neq 0$ if and only if $m=0$, and that a univalent
logharmonic mapping vanishes at the origin if and only if $m=1$, that
is, $f$ has the form\vspace{+0.1cm}
\begin{equation*}
f(z)=z|z|^{2\beta }h(z)\overline{g(z)}, \vspace{+0.1cm}
 \vspace{+0.2cm}
\end{equation*}
where $0\notin (hg)(U).$ This class has
been studied extensively over recent years in \cite{Abd1,RAli,RAli1,Abd2,Abd3,Abd4,Abd5,Abd11}.

\vspace{+0.2cm}
As further evidence of its importance, note that $F(\zeta )=\log f(e^{\zeta
})$ are univalent harmonic mappings of the half-plane $\{\zeta :{\rm Re \ }
\zeta <0\}$. Studies on univalent harmonic mappings can be found in \cite{Abu,Sheil,Dur1,Dur2,Dur3,Hen1,Hen2,Jun}, which are closely
related to the theory of minimal surfaces (see \cite{Nit,Oss}).

\vspace{+0.2cm}
Denote by $S_{Lh}$ the class consisting of univalent logharmonic mappings $f$ in $U$ with respect to some $a \in B_0$ of the form
\begin{equation*}
f(z)=z h(z)\overline{g(z)},
  \vspace{+0.1cm}
\end{equation*}
normalized by  $h(0)= 1=g(0)  ,$ and  $h$ and $g$ are nonvanishing analytic functions in $U.$
Also let $S_{Lh}^{\ast }$ denote its subclass of univalent starlike logharmonic mappings.

\vspace{+0.2cm}
An analytic function $\varphi$ in $U$ is typically real if $\varphi(z)$ is real whenever $z$ is real and nonreal elsewhere. Similarly, a logharmonic mapping $f$ in $U$ is typically real if $f(z)$ is real whenever $z$ is real and nonreal elsewhere.  Investigations into typically real logharmonic mappings was initiated by Abdulhadi in \cite{Abd3}.

\vspace{+0.2cm}
This paper treats the class $T_{Lh}$ of logharmonic mappings $f(z)=zh(z)\overline{g(z)}$ satisfying $\varphi (z)=zh(z)g(z) \in HG$ and is analytically typically real in $U$. Here $HG$ is the class of analytic functions $\varphi(z)=zh(z)g(z),$ where $h$ and $g$ in $\mathcal{H}(U)$ are normalized by $h(0)=1=g(0),$ and $0\notin (hg)(U).$
It is evident that mappings $\varphi (z)=zh(z)g(z)$ in the class $HG$ are rotations of the corresponding logharmonic mappings $f(z)=zh(z)\overline{g(z)}$.

\vspace{+0.2cm}
In Section 2, every mapping $f \in T_{Lh}$ is shown to be a product of two particular logharmonic mappings, each of which admits an integral representation. The radius of starlikeness is also obtained for the class $T_{Lh}$, as well as an upper estimate for its arclength.

\vspace{+0.2cm}
For an analytic univalent function $f(z)=z+\sum^\infty_{n=2}a_n z^n,$ it is known \cite{Abd3} that $f$ is typically real
if and only if the image $f(U)$ is a domain symmetric with respect to the
real axis. However, this characterization no longer holds for logharmonic maps, that is, it is not true that a univalent
logharmonic mapping $F(z)=zh(z)\overline{g(z)}\in T_{Lh}$ if and only if the image $F(U)$ is a
symmetric domain with respect to the real axis.
In Section 3 we explore conditions on the dilatation $a$ that
would ensure a univalent logharmonic mapping
 $f(z)=zh(z)\overline{g(z)}\in T_{Lh}$ necessarily satisfies $f(U)$ is symmetric with respect  to the real axis. Sufficient conditions for univalent logharmonic mappings to be in the class $T_{Lh} $ are also determined.

\section{An integral representation and radius of starlikeness}

The first result is to establish an integral representation for logharmonic mappings.
\begin{lemma}\label{lem1}
Let $f(z)=zh(z)\overline{g(z)}$ be a logharmonic mapping with respect to $a\in B,$ and $\varphi (z)=zh(z)g(z)$ with $h, g \in \mathcal{H}(U).$ Then\vspace{+0.1cm}
\begin{align*}
f(z)=\varphi (z)\exp \left(-2i\, {\rm Im\ }\int_{0}^{z}\dfrac{a(s)}{1+a(s)}\dfrac{\varphi
^{\prime }(s)}{\varphi (s)}ds\right).
\end{align*}
\end{lemma}

\begin{proof}
Since
\begin{align}\label{eq2.2}
f(z)=\varphi (z)\dfrac{\overline{g(z)}}{g(z)},
\end{align}
it follows from \eqref{ede} that
\begin{align*}
\frac{g^{\prime }(z)}{g(z)}=a(z)\left(\frac{\varphi'(z)}{\varphi(z) }-\frac{g'(z)}{g(z)}
\right).
\end{align*}
Thus
\begin{equation}\label{eqg}
\dfrac{g'(z)}{g(z)}=\frac{a(z)}{1+a(z)}\frac{\varphi'(z)}{\varphi(z) },\vspace{+0.2cm}
\end{equation}
which yields
\begin{equation}\label{eq2.3}
g(z)=\exp \int_{0}^{z}\dfrac{a(s)}{1+a(s)}\dfrac{\varphi ^{\prime }(s)}{%
\varphi (s)}ds.\vspace{+0.3cm}
\end{equation}
The result is readily inferred by substituting \eqref{eq2.3} into \eqref{eq2.2}.
\end{proof}

\vspace{+0.1cm}
Let $T_{Lh}^{0}$ denote the subclass of $T_{Lh}$
consisting of logharmonic mappings $q$ in $U$ with respect to $a \in B $ of the form $q(z)=zu(z)\overline{v(z)} $ and satisfying $ zu(z)v(z)=z/(1-z^{2})$.
 It follows from Lemma \ref{lem1} that\vspace{+0.2cm}
\begin{align}\label{eq2.4}
q(z)=\frac{z}{1-z^{2}}\exp \left( -2i\, {\rm Im\ } \int_{0}^{z}\frac{a(s)}{1+a(s)}
\frac{1+s^{2}}{s(1-s^{2})}ds\right).
\end{align}

\vspace{+0.4cm}
Denote by $ \mathcal{P} _{\mathbb{R}}$ the class of
normalized analytic functions with positive real part and  with real coefficients in $U.$  Further denote by  $\mathcal{P}_{Lh} $ the class consisting of logharmonic mappings $w$  with respect to $a \in B  $ of the form $w(z)=s(z)\overline{t(z)},$ where $s,t \in \mathcal{H}(U) $ are normalized by $  s(0) = 1=t(0), $ and satisfy
$p(z) =s(z)t(z)\in  \mathcal{P} _{\mathbb{R}} .$
 Similar to the proof of Lemma \ref{lem1},
it is readily established that\vspace{+0.1cm}
\begin{equation}\label{eq2.5}
w(z)=p(z)\exp \left(-2i\, {\rm Im\ } \int_{0}^{z}\frac{a(s)}{1+a(s)}\dfrac{p^{\prime }(s)}{%
p(s)}ds\right).\vspace{+0.3cm}
\end{equation}
Note that the class $\mathcal{P}_{Lh}$ also contains the set $\mathcal{P} _{\mathbb{R}}.$

\vspace{+0.3cm}
The following result gives a representation formula for functions in the class $T_{Lh}$ in terms of functions in $T_{Lh}^0$ and $\mathcal{P}_{Lh}.$

\begin{theorem}\label{thm1}
A function $f$ belongs to $T_{Lh}$ with respect to  $a\in B $ if and only if  $f(z)=q(z)w(z) $  for some $q$ belonging to $ T_{Lh}^{0}$ with respect to  $a\in B $ and $w\in \mathcal{P}_{Lh}$ with respect to  $a\in B.$
\end{theorem}

\begin{proof}
Let \ $f(z)=zh(z)\overline{g(z)}\in T_{Lh}$ with respect to $a\in B.$ It is known \cite{Rog} that every typically real analytic function $\varphi$ has the form   $(1-z^2)\varphi(z)=zp(z)$ for some $p \in \mathcal{P}_{\mathbb{R}}$. Thus Lemma \ref{lem1} yields\vspace{+0.2cm}
\begin{align*}
f(z) &=\frac{zp(z)}{1-z^{2}}\exp \left( -2i{\rm Im\ } \int_{0}^{z}\frac{a(s)}{1+a(s)}%
\left( \frac{1+s^{2}}{s(1-s^{2})}+\frac{p'(s)}{p(s)}\right) ds \right)
\\&=\left(\frac{z}{1-z^{2}}\exp \left( -2i{\rm Im\ } \int_{0}^{z}\frac{a(s)}{
1+a(s)}\frac{1+s^{2}}{s(1-s^{2})}ds\right) \right)
\\& \quad \quad \times \left( p(z)\exp \left(
-2i{\rm Im\ } \int_{0}^{z}\dfrac{a(s)}{1+a(s)}\frac{p^{\prime }(s)}{p(s)}ds\right)
\right)
\\&:=q(z)w(z),
\end{align*}
where from \eqref{eq2.4} and \eqref{eq2.5},
 \[
q(z)=\frac{z}{1-z^{2}}\exp \left( -2i{\rm Im\ } \int_{0}^{z}\frac{a(s)}{%
1+a(s)}\frac{1+s^{2}}{s(1-s^{2})}ds\right) \in T_{Lh}^{0},
\] and\vspace{-0.2cm}
\[w(z)=p(z)\exp \left( -2i{\rm Im\ } \int_{0}^{z}\dfrac{a(s)}{1+a(s)}\dfrac{p^{\prime
}(s)}{p(s)}ds\right) \in  \mathcal{P}_{Lh}.\vspace{+0.3cm}\]

\vspace{+0.2cm}
Conversely, if $f(z)=q(z)w(z) =zh(z)\overline{ g(z)}  ,$ then
 \eqref{eq2.4} and \eqref{eq2.5} yield
 \[h(z)=\frac{p(z)}{(1-z^{2})}\exp \Bigg(  -  \int_{0}^{z}\frac{a(s)} {1+a(s)}\bigg( \frac{1+s^{2}}{s(1-s^{2})}+ \dfrac{p^{\prime
}(s)}{p(s)}\bigg) \Bigg)ds,\vspace{+0.2cm}
\]
and
\[g(z)=\exp    \int_{0}^{z}\frac{a(s)}{%
1+a(s)}\bigg( \frac{1+s^{2}}{s(1-s^{2})}+ \dfrac{p^{\prime
}(s)}{p(s)}\bigg)ds .\vspace{+0.5cm} \]
Thus $\varphi(z)=zh(z)g(z)=zp(z)/(1-z^2),$ $ p \in \mathcal{P}_{\mathbb{R}}.$ It follows from \cite{Rog} that $\varphi \in T,$ and hence $f \in T_{Lh}.$
\end{proof}

\vspace{+0.1cm}
\begin{corollary}\label{cor1}
If $f $ belongs to $T_{Lh}$ with respect to  $a\in B, $ then  $q^2(z)/f(z)\in T_{Lh} $ for some $q\in T_{Lh}^{0}$ with respect to the same $a\in B. $
\end{corollary}
\begin{proof}
It follows from Theorem \ref{thm1} that $f(z) =q(z) w(z),$ where  $q\in T_{Lh}^{0}$ and $w(z) =s(z)\overline{t(z)} \in\mathcal{P}_{Lh}.$ Since
\begin{align*}
\overline{\Bigg(\frac{ \big(\frac{1}{w }\big)_{\overline{z}}}{   \frac{1}{w } }\Bigg)}
 = \overline{\Bigg(\frac{ -(w ) _{\overline{z}}}{w  } \Bigg)}
=\frac{-a (w)_{z}}{w}=a \frac{\left(\frac{1}{w} \right)_{z}}{\frac{1}{w}},
\end{align*}
it follows that $1/w$ is logharmonic with respect to the same $a.$ Furthermore, $ 1/w   = 1/(s \overline{t })=( 1/ s ) (\overline{  1/t  }) ,$ and
\begin{align*}
{\rm Re \ } \left(\frac{1}{s t }\right)&={\rm Re \ } \left(\frac{\overline{ s t }}{|s t |^2}\right)=\frac{{\rm Re \ }(  s t  )}{|s t |^2}>0.
\end{align*}
Also, $1/w$ has real coefficients. Thus the function $1/w \in \mathcal{P}_{Lh}$. Hence Theorem \ref{thm1} shows that $ q/w = q^{2}/f \in T_{Lh}.$
\end{proof}

\vspace{+0.1cm}
The next result obtains an estimate for the radius of starlikeness for the class $T_{Lh}$.
\begin{theorem}\label{thm2}
Let $f(z)=zh(z)\overline{g(z)}\in T_{Lh} .$
Then $f$ maps the disk $|z|<3-2\sqrt{2}$ onto a starlike
domain.\end{theorem}

\begin{proof}
The function $f$ maps the circle $|z|=r$ onto a starlike curve provided
\[\frac{\partial}{\partial \theta} \arg f(re^{i\theta})   =  {\rm Im  } \left(\frac{\partial}{\partial \theta} \log f(re^{i\theta})\right)
={\rm Re\ } \frac{zf_{z}-\overline{z}f_{\overline{z}}}{f} >0.\vspace{+0.2cm}\]
With $\varphi(z)=zh(z)g(z),$ a short computation gives\vspace{+0.2cm}
\begin{equation*}
{\rm Re\ } \frac{zf_{z}-\overline{z}f_{\overline{z}}}{f}= {\rm Re\ } \left( \frac{1-a(z)}{1+a(z)} \frac{z\varphi'(z)}{\varphi(z)} \right)
\end{equation*}
for some $a \in B.$

\vspace{+0.2cm}
Next let\vspace{+0.1cm}
\[q(z)=\frac{1-a(z)}{1+a(z)} \frac{z\varphi'(z)}{\varphi(z)},\vspace{+0.3cm}\]
and $\sigma(z)=\rho_{0}z.$  Kirwan \cite{Kir} has shown that the radius of starlikeness for typically real analytic functions $\varphi$ is $\rho_{0}=\sqrt{2}-1.$ Thus ${\rm Re\ } \varphi(\sigma(z))>0, $
and so $q(\sigma(z))$ is subordinated to $((1+z)/(1-z))^{2}$ in $U$.

\vspace{+0.15cm}
Writing $p(z)=(1+z)/(1-z),$ it follows from \cite[p.\ 84]{goodman} that\vspace{+0.15cm}
\begin{equation*}
\left|p(z)-\frac{1+r^2}{1-r^2}\right| \leq \frac{2r}{1-r^2}.\vspace{+0.1cm}
\end{equation*}%
Thus $|\arg(p(z))|  < \pi/4$ provided $|z| < \rho_{0},$ where $\rho_{0}$  is a smallest positive root of the equation $ r^2-2\sqrt{2} r+1=0.$ The function $f(z)=zh(z)\overline{g(z)}$ is thus
starlike in the disk $|z|<\rho_{0}^2=3-2\sqrt{2}.$
\end{proof}

\vspace{+0.2cm}
In the next result, an upper estimate is established for arclength of
all mappings $f$ in the class $T_{Lh}.$

\begin{theorem}\label{thm3}
Let $f(z)=zh(z)\overline{g(z)}\in T_{Lh},$
and $|f(z)|\leq M(r)$, $0<r<1$.
Then an upper bound for its arclength $L(r)$ is given by
\begin{equation*}
L(r)\leq 4\pi M(r)\dfrac{1+r+2r^2-2r^3}{(1-r)(1-r^{2})}.\vspace{+0.2cm}
\end{equation*}
\end{theorem}

\begin{proof}
Let $C_{r}$ denote the image of the circle $|z|=r<1 $ under the mapping $w=f(z)$. Then
\begin{align*}\label{eq2.9}
L(r) &=\int_{C_{r}}|df|\ =\int_{0}^{2\pi }|zf_{z}-\overline{z}f_{\overline{z%
}}|d\theta  \notag \\
&\leq M(r)\int_{0}^{2\pi }\left\vert \dfrac{zf_{z}-\overline{z}f_{\overline{%
z}}}{f}\right\vert d\theta .
\end{align*}

\vspace{+0.2cm}
Since $\varphi (z)=zh(z)g(z)=zp(z)/(1-z^{2})$ for some $p \in \mathcal{P} _{\mathbb{R}},$  it follows that\vspace{+0.2cm}
\[\frac{zf_{z}-\overline{z}f_{\overline{z}}}{f}={\rm Re\ } \left( \frac{1-a(z)}{%
1+a(z)}\left( \frac{zp'(z)}{p(z)}+\frac{1+z^{2}}{1-z^{2}}\right) %
\right) +i{\rm Im\ } \left( \frac{zp'(z)}{p(z)}+\frac{1+z^{2}}{1-z^{2}}%
\right) .\vspace{+0.2cm}\]
\newpage
Therefore,\vspace{+0.2cm}
\begin{align*}
\frac{L(r)}{M(r)} &\leq \int_{0}^{2\pi }\left\vert {\rm Re\ }\bigg( \dfrac{1-a(z)}{1+a(z)}\left(
\dfrac{zp^{\prime }(z)}{p(z)}\right)\bigg)\right\vert d\theta \notag \\
&\quad \quad +\int_{0}^{2\pi }\left\vert {\rm Re\ }\bigg( \dfrac{1-a(z)}{1+a(z)}\left(\dfrac{1+z^{2}}{1-z^{2}}\right)\bigg) \right\vert
d\theta  \notag \\
&\quad \quad +\int_{0}^{2\pi }\left\vert {\rm Im\ } \dfrac{zp^{\prime }(z)}{p(z)}%
\right\vert d\theta + \int_{0}^{2\pi }\left\vert {\rm Im\ } \dfrac{1+z^{2}}{%
1-z^{2}}\right\vert d\theta,  \notag
\end{align*}
that is,
\begin{align}\label{eq2.10}
L(r)\leq M(r)\left(I_{1}+I_{2}+I_{3}+I_{4}\right).
\end{align}

\vspace{+0.2cm}
The function $p$ is subordinate to $(1+z)/(1-z),$ and thus $zp'(z)/p(z)=2zw'(z)/(1-w^{2}(z))$ for some analytic self-map $w$ of $U$ with $w(0)=0$. The Schwarz-Pick inequality states
\[\frac{|w'(z)|}{1-|w(z)|^2} \leq \frac{1}{1-|z|^2}.\] Thus
\begin{align*}\label{eq2.11}
I_{1} &=\int_{0}^{2\pi }\left\vert {\rm Re\ } \left(\dfrac{1-a(z)}{1+a(z)}\bigg(
\dfrac{zp^{\prime }(z)}{p(z)}\bigg)\right)\right\vert d\theta \notag \\
&\leq \int_{0}^{2\pi }\left\vert \dfrac{1+z}{1-z}\right\vert \left\vert \dfrac{2zw'(z)}{1-w^{2}(z)}\right\vert
d\theta \notag
 \leq \dfrac{1+|z|}{1-|z|} \int_{0}^{2\pi }\dfrac{2|z|}{1-|z|^{2}}d\theta \notag \\
&= \dfrac{4\pi r}{(1-r)^2}.
\end{align*}

\vspace{+0.2cm}
Since $[(1-a(z)/(1+a(z))][(1+z^{2})/(1-z^{2})]$ is subordinated to $((1+z)/(1-z)
)^{2},$ it follows from Parseval's theorem that
\begin{align*}
I_{2}&=\int_{0}^{2\pi }\left\vert {\rm Re\ }\left( \dfrac{1-a(z)}{1+a(z)}\bigg(\dfrac{1+z^{2}}{1-z^{2}}\bigg)\right) \right\vert
d\theta \notag\\
 &\leq \int_{0}^{2\pi }\left\vert \ \left( \dfrac{1+z}{1-z}\right)
^{2}\right\vert d\theta \leq 2\pi \left( 1+4\underset{n=1}{\overset{\infty }{%
\sum }}r^{2n}\right)  \notag \\
&= 2\pi \left( \dfrac{1+3r^{2}}{1-r^{2}}\right),
\end{align*}

Also,
\begin{align*}
I_{3} &=\int_{0}^{2\pi }\left\vert \ {\rm Im\ } \dfrac{zp'(z)}{p(z)}%
\right\vert d\theta \leq \int_{0}^{2\pi }\left\vert \dfrac{2zw'(z)}{1-w^{2}(z)}\right\vert
d\theta \notag \\
 &\leq  \int_{0}^{2\pi }\dfrac{2|z|}{1-|z|^{2}}d\theta
  =\dfrac{4\pi r}{1-r^{2}},
\end{align*}

and
\begin{align*}
I_{4}&=\int_{0}^{2\pi }\left\vert {\rm Im\ } \dfrac{1+z^{2}}{1-z^{2}}\right\vert
d\theta \leq  \int_{0}^{2\pi }\left\vert \dfrac{1+z^{2}}{1-z^{2}}\right\vert
d\theta \notag \\
 &\leq 2\pi \dfrac{1+r^{2}}{1-r^{2}}.
\end{align*}
Substituting the bounds for \ $I_{1},$ $\ I_{2},$  $\ I_{3}$ and $I_{4}$ into \eqref{eq2.10} yields\vspace{+0.1cm}
\begin{align*}
L(r)\leq 4\pi M(r)\dfrac{1+r+2r^2-2r^3}{(1-r)(1-r^{2})}.
\end{align*}
\end{proof}

\section{Univalent logharmonic mappings in the class $T_{Lh}$}

As noted earlier, an analytic univalent function $\varphi$ is typically real  if and only if the image $\varphi(U)$ is a domain symmetric
with respect to the real axis. Such a geomteric characterization no longer holds true for the class $T_{Lh}$. The following example illustrates a univalent logharmonic mapping $F(z)=zh(z)\overline{g(z)}\in T_{Lh}$ where $F(U)$ is not a symmetric domain.

\begin{example}\label{exam1}
Let
\begin{equation*}
F(z)=zh(z)\overline{g(z)}=z\left(1+\frac{iz}{3}\right ) \left(1+\frac{i\overline{z}}{3}\right ).\vspace{+0.2cm}
\end{equation*}
It is evident that $F(0)=0,$ $h(0)= 1=g(0),$ where $h(z)=1+ iz/3 $ and $g(z)=1-i z/3 .$ Also,
\begin{equation*}
|a(z)|=\left\vert F\overline{F_{\overline{z}}}%
/F_{z}\overline{F}\right\vert  =\left\vert
\dfrac{-iz\left( 3+iz\right) }{(3-iz)(3+2iz)}\right\vert <1.\vspace{+0.2cm}
\end{equation*}
Thus $F$ is a normalized logharmonic mapping in $U $ with respect to $  a\in B_0.$

\vspace{+0.1cm}
Let
\begin{equation*}
 \psi(z)= \frac{zh(z)}{g(z)}=\frac{z\left(1+\frac{iz}{3}\right)}{\left(1-\frac{iz}{3}\right)}.
 \vspace{+0.1cm}
\end{equation*}
Then $\psi(0)=0,$ $\psi'(0)=1,$ and\vspace{+0.2cm}
 \begin{equation*}
 {\rm Re\ }\left(\frac{z\psi'(z)}{\psi(z)} \right)={\rm Re\ }\left(\frac{  z^2+6iz+9}{ z^2+9} \right) >
 \frac{10}{(9+r^2\cos 2\theta)^2+r^4\sin ^2 {2\theta}}>0.\vspace{+0.25cm}
 \end{equation*}
Hence $\psi\in S^{\ast },$  and thus \cite{Abd5} shows that  $F$ is in fact a univalent starlike logharmonic mapping.

\vspace{+0.2cm}
Next, let $\varphi (z)=zh(z)g(z)=z\left( 1+z^{2}/9\right) .$ Evidently, for $z=x+iy,$\vspace{+0.1cm}
\begin{equation*}
 {\rm Im } (\varphi (z))
 =\frac{y}{9} \left(3x^2 +9-y^2\right). \vspace{-0.1cm}
\end{equation*}
Then
\begin{equation*}
{\rm Im }(z){\rm Im } (\varphi (z))
  =y^{2}\left(\dfrac{3x^{2}+(9-y^{2})}{9}\right)>0\vspace{+0.2cm}
\end{equation*}
whenever ${\rm Im\ } (z)\neq 0.$ Hence $\varphi$ is typically real, and thus $F \in T_{Lh}.$

\vspace{+0.2cm}
 A simple calculation shows that\vspace{+0.1cm}
 \begin{equation*}
F(z)
 =z\left(1+\frac{2i }{3}  {\rm Re\ }z-\frac{|z|^2}{9}\right ).\vspace{-0.1cm}
\end{equation*}
With
\begin{equation*}
 I(t)={\rm Im\ } F(e^{it})=\frac{ 2\cos ^2 t}{3}+\frac{8}{9}\sin t,
\end{equation*}
then
\begin{equation*}
 I'(t)=\frac{4}{3}\cos t\left( \frac{2}{3}-\sin t\right).\vspace{+0.2cm}
\end{equation*}
It follows that $I'(t)=0$ when $t=\pm \pi/2,$ or $t=\sin^{-1} (2/3)$. Thus
 \begin{equation*}
 M=\max_{|t|\leq \pi}I(t)=I\left(\sin^{-1}\left( \frac{2}{3}\right )\right )=\frac{26}{27},
 \end{equation*}
 and
 \begin{equation*}m=\min_{|t|\leq \pi}I(t)=I(-\frac{\pi}{2})=-\frac{ 8}{9},\vspace{+0.2cm}
 \end{equation*}
which shows that $F(U)$ is not symmetric with respect to the real axis.\end{example}

 Figure \ref{pic1} shows the mapping $F(z)=z(1+iz/3)(1+i\overline{z}/3)$ which is not symmetric with respect to the real axis.
\begin{figure}[h]
\centering
\begin{tabular}{cc}
\includegraphics[width=7cm,height=7cm]{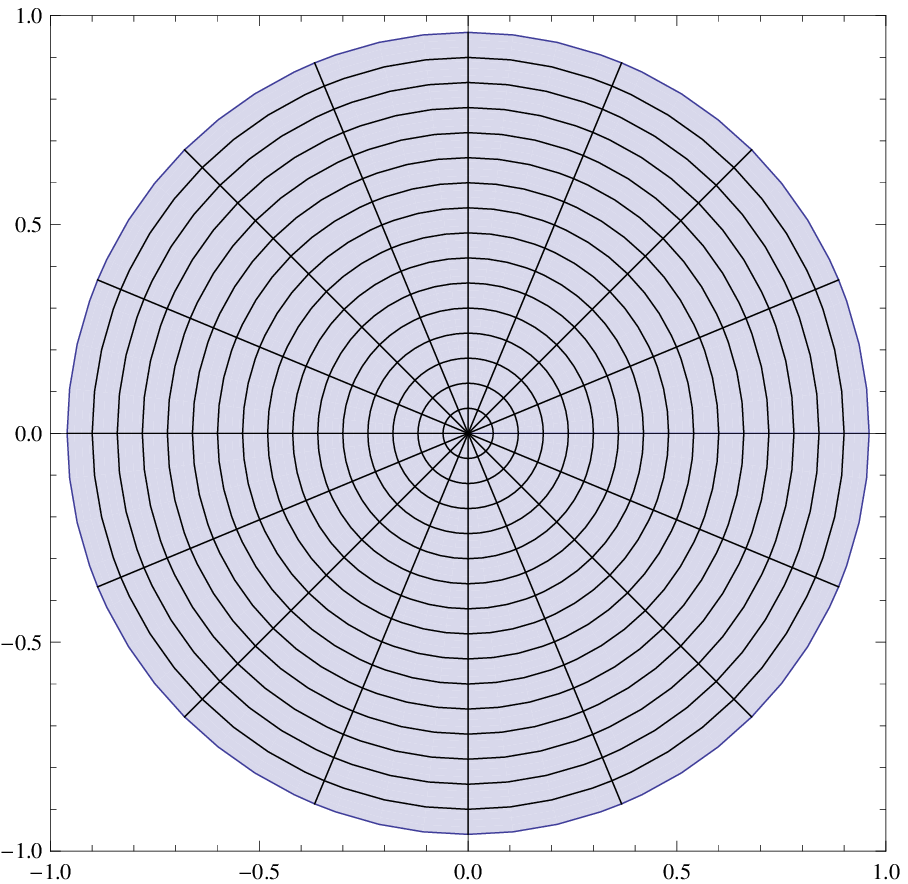} &
\includegraphics[width=7cm,height=7cm]{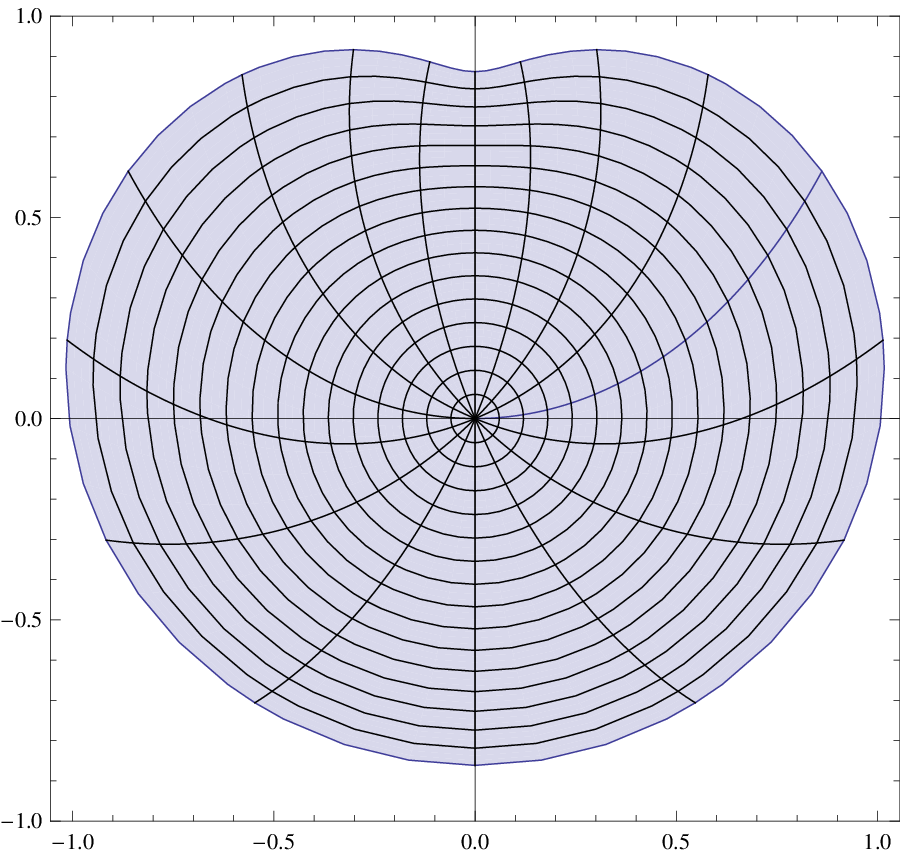}
\end{tabular}
\caption{Graph of $F(z)=z(1+\frac{iz}{3})(1+\frac{i\overline{z}}{3}).$ }
\label{pic1}
\end{figure}

\newpage
Our next example illustrates a univalent logharmonic mapping from $U$
onto a symmetric domain $\Omega $ but does not belong to the class $T_{Lh}.$

\begin{example}\label{exam2}
Consider the function\vspace{+0.2cm}
\begin{equation*}
 F(z)=z\dfrac{1-\overline{z}}{1-z}\exp \left\{ {\rm Re\ } \left( \dfrac{4z}{1-z}\right) \right\}.\vspace{+0.2cm}
 \end{equation*}
Then $F(0)=0,$ $h(0)= 1=g(0),$ where \vspace{+0.2cm}
 \begin{equation*}
   h(z)=\frac{\exp \left\{\frac{2z}{1-z} \right\} }{1-z} , \quad \text{and}  \quad    g(z)=\exp \left\{\frac{2z}{1- z}  \right\}(1-z). \vspace{+0.2cm}
    \end{equation*}
    Also, $|a(z)|=\left\vert F\overline{F_{\overline{z}}} /F_{z}\overline{F}\right\vert=\left\vert z\right\vert <1.$
Thus
$F$ is a normalized logharmonic mapping with respect to $a $ from $U$ onto
$ \mathbb{C}\backslash(-\infty ,-1/e^{2}].$

\vspace{+0.2cm}
Let
\begin{equation*}
 \psi(z)= \frac{zh(z)}{g(z)} =\frac{z}{(1-z)^2}.\vspace{+0.3cm}
\end{equation*}
Then
  $\psi\in S^{\ast },$ and hence it follows from \cite{Abd5} that $F$ is in fact a univalent starlike logharmonic mapping. It is also evident that $F$ has real coefficients, that is, $F(z)=\overline{ F(\overline{z})},$ and so $F(U)$  is a symmetric domain  with respect to the real axis.

\newpage
Let  $\varphi (z)=zh(z)g(z)=z\exp \left\{ 4z/(1-z)\right\} .$ At $z_1=(1-2/\pi)+2i/\pi,$\vspace{+0.2cm}
\begin{align*}
{\rm Re }\left(\frac{1-z_1^2}{z_1}\varphi(z_1)\right)
&={\rm Re }\left((1-z_1^2)\exp \left\{\frac{4z_1}{1-z_1} \right\}\right)\\
&= -\frac{4}{\pi}\exp \left\{ 4-\pi  \right\}<0.
\end{align*}
 Thus $(1-z^2)\varphi(z)/z  \notin \mathcal{P}_{\mathbb{R}},$ and it follows from  \cite{Rog} that the function $\varphi$ is not typically real.\\

 Figure  \ref{pic2} shows the mapping $F(z)=z \exp \left\{ {\rm Re\ } \left( 4z/(1-z)%
\right) \right\}(1-\overline{z})/(1-z) $  which does not belong to the class $T_{Lh},$ but yet maps  $U$
 onto a symmetric domain with respect to the real axis $F(U).$\end{example}

\begin{figure}[h]
\centering
\begin{tabular}{cc}
\includegraphics[width=7cm,height=7cm]{Figure3a.eps} &
\includegraphics[width=7cm,height=7cm]{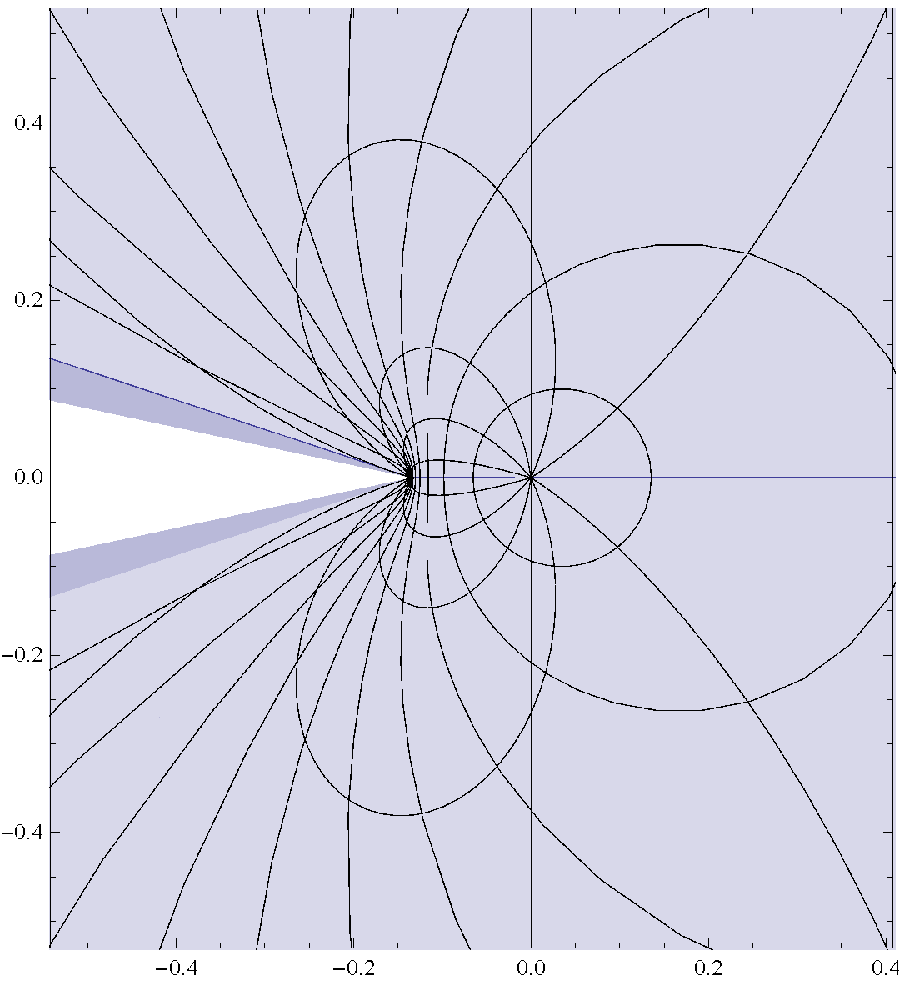}
\end{tabular}
\caption{Graph of $ F(z)=z\dfrac{1-\overline{z}}{1-z}\exp \left\{ {\rm Re\ } \left( \dfrac{4z}{1-z}%
\right) \right\}.$}
\label{pic2}
\end{figure}

The following result describes the geometry of a univalent logharmonic function in the class $T_{Lh}$ when its second dilatation has real coefficients.

\begin{theorem}\label{thm64}
Let $f(z)=zh(z)\overline{g(z)}\in T_{Lh}$ be a univalent sense-preserving logharmonic mapping in $U$. If the second dilatation function $a$ has real coefficients, that is, $a(\overline{z})=\overline{a(z)},$ then $f(U)$ is symmetric with respect to the real axis.
\end{theorem}

\begin{proof}
Let $\varphi (z)=zh(z)g(z)$ be analytically typically real.
Then $\varphi$ has
real coefficients. It follows from \eqref{eqg} that\vspace{+0.2cm}
\begin{equation*}
\frac{g'(z)}{g(z)}=\frac{a(z)}{1+a(z)}\frac{\varphi'(z)}{\varphi (z)},\vspace{+0.2cm}
\end{equation*}
which readily yields the solution $g$. Thus $g$ has real coefficients. It is also evident that $h$ has real coefficients since  $h (z)=\varphi(z)/zg(z).$ Therefore, $f(z)=zh(z)\overline{g(z)}$ has real coefficients, whence $f(U)$ is symmetric with respect to the real
axis.
\end{proof}

The final result derives sufficient conditions for $f$ to be typically real logharmonic in some subdisk of $U$.

\begin{theorem}\label{thm65}
Let $f(z)=zh(z)\overline{g(z)}$ be a univalent sense-preserving logharmonic mapping in $U$ normalized by $h(0)=1=g(0),$ with its second dilatation function $a$ has real coefficients. Further, suppose  $f(U)$ is a strictly starlike Jordan domain, that is, each radial ray from $0$ intersects the boundary $\partial \Omega$ of $\Omega=f(U)$ in exactly one point of $\mathbb{C}.$ If $f(U)$ is a symmetric domain with
respect to the real axis, then $\varphi (z)=zh(z)g(z)$ is typically real  in the disk $|z|<\sqrt{2}-1.$
\end{theorem}
\begin{proof}
Suppose $f(U)$ is a strictly starlike Jordan domain symmetric with respect
to the real axis, and let $F(z) = \overline{f(\overline{z})} ,$ where
$f(z)=zh(z)\overline{g(z)}$ is a univalent logharmonic mapping.
Then $F $ is univalent in $U $ and $ F(U) =  f (U).$

\vspace{+0.1cm}
Let $F(z)=z H(z)\overline{G(z)}  =  z \overline{h(\overline{z})} g(\overline{z})$ with $H(z)=\overline{h(\overline{z})}$ and $G(z)=\overline{g(\overline{z})}.$
Thus $F(0) = 0,$ and $H(0) = 1 =G(0).  $
With $ a^*(z) =  F\overline{F_{\overline{z}}}%
/F_{z}\overline{F} ,$ then\vspace{+0.1cm}
{\large \begin{equation*}
 a^*(z)
 =\frac{\overline{\Big ( \frac{ ( g  (\overline{ z})  )_{\overline{ z}}}{g (\overline{ z})} \Big )}}{\frac{1}{z}+
  \frac {\left(\overline{h (\overline{ z})}\right) _{ z}}
 {\overline{h  (\overline{ z})}}}
    =\frac{  \frac{\left(\overline{g  (\overline{ z})}\right)_z}{\overline{g (\overline{ z})} } }{\frac{1}{z}+
   \frac {\left(\overline{h (\overline{ z})_{ \overline{ z}}}\right) }
 {\overline{h  (\overline{ z})}}}
 =\overline{a(\overline{z})}.\vspace{+0.2cm}
\end{equation*}}
Since $a$ has real coefficients, it is evident that $ a^*(z)=a(z).$
Therefore,  $F$ is a logharmonic mapping with respect to the same $a.$ Also, $H(0) = h(0)=1.$ It then follows from {\rm \cite[Lemma 2.4]{Abd11 }} that there is
only one univalent logharmonic mapping from $U$ onto $f(U)$ which is a solution of \eqref{ede}
normalized by $f(0) = 0 $ and $h(0) =1 =g(0) .$ In other words,
  $f(z)=F(z)=\overline{f(\overline{z})},$
 and thus $f$ has real coefficients. Hence $\psi (z)=zh(z)/g(z)=f(z)/|g(z)|^2$ has real coefficients.

\newpage
Direct calculations yield\vspace{+0.2cm}
\begin{equation*}
\frac{g'(z)}{g(z)}=\frac{a(z)}{1-a(z)}\frac{\psi'(z)}{\psi(z) },\vspace{+0.2cm}
\end{equation*}%
which upon integrating leads to\vspace{+0.2cm}
\begin{equation*}
g(z)=\exp \int_{0}^{z}\frac{a(t)}{1-a(t)}\dfrac{\psi ^{\prime }(t)}{\psi (t)}dt.
\end{equation*}%
Then $g$, and so does $h$, have real coefficients, and thus $\varphi (z)=zh(z)g(z)$ has real coefficients.
Furthermore,   {\rm \cite[Theorem 3.1]{RAli}} shows that $\varphi $ is starlike univalent
in the disk $|z|<\rho,$ where $\rho=\sqrt{2}-1 .$ Thus $\varphi $ is typically real  in the disk $|z|<\sqrt{2}-1.$
\end{proof}

 \noindent \textbf{Acknowledgment.} The third author gratefully acknowledged
 support from Universiti Sains Malaysia research grant
 1001/PMATHS/811280.

\end{document}